\documentclass[twoside,a4paper,11pt,intlimits]{elsarticle}
\usepackage{amsmath}
\usepackage{amsthm}
\usepackage{amssymb}
\usepackage{amsfonts}
\usepackage{amsxtra}

\usepackage[scrtime]{prelim2e}
\renewcommand{\PrelimText}{%
  \textnormal{\texttt{/\jobname},
    \qquad\today,\quad\thistime}}

\usepackage{enumitem}
\setenumerate{label={\rm (\alph{*})}}

\usepackage{color}
\usepackage{srcltx}

\usepackage[rose,frak,operator]{paper_diening}

\usepackage{graphicx}

\newtheorem{theorem}{Theorem}[section]
\newtheorem{lemma}[theorem]{Lemma}
\newtheorem{proposition}[theorem]{Proposition}
\newtheorem{corollary}[theorem]{Corollary}

\newtheorem{remark}[theorem]{Remark}

\numberwithin{equation}{section}

\newcommand{\meantmp}[2]{#1\langle{#2}#1\rangle}
\newcommand{\mean}[1]{\meantmp{}{#1}}

\newcommand{\linear}[2]{{\ell_{#2}(#1)}}

\newcommand{\RNn}{{\setR^{N \times n}}}

\newcommand{\dt}{\ensuremath{\,{\rm d} t}}
\newcommand{\dtau}{\ensuremath{\,{\rm d} \tau}}
\newcommand{\dx}{\ensuremath{\,{\rm d} x}}
\newcommand{\dy}{\ensuremath{\,{\rm d} y}}
\newcommand{\dz}{\ensuremath{\,{\rm d} z}}

\newcommand{\osc}[2]{\ensuremath{\text{osc}_{{#2}}(#1)}}
\newcommand{\Qrho}[1]{{Q^{\lambda_{#1\rho}}_{#1\rho}}}
\newcommand{\Qr}[1]{{Q^{\lambda_{#1 r}}_{#1 r}}}

\newcommand{\normg}{\norm{g}_{L^\infty(I,\setBMO(B_r))}}
\newcommand{\normgw}[1]{\norm{g}_{L^\infty(I,\setBMO_{\omega'}(B_{#1}))}}

\newcommand{\setBLO}{\text{BLO}}

\begin{document}
\begin{frontmatter}
 
\title{H\"older-Zygmund Estimates for Degenerate Parabolic Systems}

\author{Sebastian Schwarzacher}
\ead{schwarz@math.lmu.de}

\address{LMU Munich, Institute of Mathematics, Theresienstr. 39,
  80333-Munich, Germany}

\begin{abstract}
 We consider energy solutions of the inhomogeneous parabolic $p$-Laplacien system
$\partial_t u-\divergence (\abs{\nabla u}^{p-2}\nabla u)=-\divergence g)$. We show in the case $p\geq 2$ that if the right hand side $g$ is locally in $L^\infty(\setBMO)$, then $u$ is locally in $L^\infty(\mathcal{C}^1)$, where $\mathcal{C}^1$ is the 1-H\"older--Zygmund space. This is the borderline case of the \Calderon-Zygmund theorey. 
We provide local quantitative estimates. We also show that finer properties of $g$ are conserved by $\nabla u$, e.g.\ H\"older continuity. 
 Moreover, we prove a new decay for gradients of $p$-caloric solutions for all $\frac{2n}{n+2}<p<\infty$.  
\end{abstract}
%
%
%


\end{frontmatter}
\section{Introduction}
\noindent
We study local behavior of solutions $u:Q_T\to\setR^N$ to the inhomogeneous parabolic p-Laplace system.
\begin{align}
 \label{eq}
\partial_t u-\Delta_p u=\partial_tu-\divergence(\abs{\nabla u}^{p-2}\nabla u)=-\divergence g.
\end{align}
If $g\in L^{p'}(Q_T)$ this problem is well-posed and local solutions exist; here $Q_T$ is a space time cylinder. Solutions with this type of term on the right hand side are called energy solutions.
 It is the aim of the non-linear \Calderon-Zygmund theory to transfer information from $g$ to  $\nabla u$, the gradient of the solution. The theory started with the important paper of Iwaniec~\cite{Iwa83}. In this article the elliptic p-Laplace is considered. It states that if $g\in L^{p'q}$ for $1\leq q<\infty$, then $\nabla u\in L^{pq}$. In \cite{AceMin07} the same was proved for the parabolic p-Laplace \eqref{eq} including local estimates.
 On the other hand Misawa~\cite{Mis02} proved that if $g$ is H\"older continuous, then $\nabla u$ is H\"older continuous for conveniently small H\"older exponents. Later this result was refined and extended by Kuusi and Mingione~\cite{KuuMin12} (see also \cite{Mis13}). It is the concern of this article to close the gap between higher integrability and H\"older continuity, especially the limit case $q=\infty$. Even in the linear elliptic setting (i.e.\ Poisson's equation) we know that $g\in L^\infty$ does not imply $\nabla u\in L^\infty$. As in this case where $f\mapsto \nabla u$ is a singular integral operator, the right limit space is the space of bounded mean oscillation ($\setBMO$). In the case of the (non-linear) elliptic p-Laplacian the right limit space is the same. Indeed, it was shown in \cite{DiBMan93} and \cite{DieKapSch11} that $g\in \setBMO$ implies  $\abs{\nabla u}^{p-2}\nabla u\in\setBMO$ (locally).  The task to find a satisfactory limit space in the parabolic setting turns out to be difficult. We will introduce this matter by looking at the inhomogeneous heat equation. For the linear theory we have the natural space of parabolic bounded mean oscillation.
 We say that $f\in\setBMO_{\text{par}}(\Omega)$, if $f\in L^1(\Omega)$ and
\[
 \norm{f}_{\setBMO_{\text{par}}(\Omega)}:=\sup\limits_{Q_{r^2,r}\subset\Omega}\dashint_{Q_{r^2,r}}\abs{f-\mean{f}_{Q_{r^2,r}}}\dz<\infty.
\]
If $p=2$, then we find that $g\in \setBMO_{\text{par}}(Q_T)$ implies $\nabla u\in \setBMO_{\text{par}}(Q_T)$.

The non-linear version of this result is the boundedness over mean oscillation of the so called natural scaled cylinders:
$Q_{\lambda^{2-p}r^2,r}=:Q^\lambda_r$, where 
\begin{align}
\label{eq:subintr}
 \lambda^p\geq \dashint_{Q_r^\lambda}\abs{\nabla u}^p\dz.
\end{align}
We carefully construct cubes of the above type and are able to bound the mean oscillations of $\nabla u$ over these natural scaled cylinders for $p\geq 2$: see Proposition~\ref{pro:intrbmo}. However, these oscillation estimates are not very  satisfactory. They depend very strongly on the solution itself. We will overcome this by proving some Bochner estimates. To motivate this result, we want to mention a result on which we worked simultaneously to this paper. There we prove $\abs{g}^{p'}\in L^\infty(I,L^q(B))$ implies $\abs{\nabla u}^p\in L^\infty(I,L^q(B))$ (locally). 
 If one let $q\to\infty$ on this quantity we realize that the right borderline space should be a Bochner space of type $L^\infty(I,X)$. The first guess is of course $X=\setBMO(B)$. It turns out that this space is too small. Instead we obtained the following main theorem.
\begin{theorem}
 \label{thm:m}
Let $u$ be a solution on $I\times B$, for $p\geq 2$. If $g\in L^\infty(I,\setBMO(B))$, then $u\in L^\infty_\loc(I,\mathcal{C}^1_\loc(B))$. Moreover, for every parabolic cylinder $Q_{2r}\subset I\times B$ 
\[
 \norm{u}_{L^\infty(I_{r^2},\mathcal{C}^1(B_r))}\leq c\norm{g}_{L^\infty(I,\setBMO(B))}^\frac1{p-1}+c\norm{\nabla u}_{L^p(Q_{2r})}+c,
\]
where the constant $c$ only depends on $n,N,p$.
\end{theorem}
Here $\mathcal{C}^1$ is the 1-H\"older-Zygmund space (see \cite{Zyg45} and Section~\ref{sec:pre} for the exact definition). It is a known substitute for $C^1$ in the setting of PDE's. To fortify this we mention the following order of spaces on a bounded set $B\subset\setR^n$
\[
 C^1(B)\subset W^{1,\setBMO}(B)\subset \mathcal{C}^1(B)\subset \bigcap_{1\leq q<\infty}W^{1,q}(B).
\]
All estimates can be found in Triebel's book \cite{Tri92}.
 The difference between these spaces and details will be discussed in Section~\ref{sec:pre} and interpolation estimates, that follow from our estimates can be found in Remark~\ref{rem:triebel}.

Theorem~\ref{thm:m} is the limit case which has not been proven before. To the authors knowledge these estimates are new even for the linear case $p=2$. Our estimates are general enough so that we can go beyond. Indeed, all our estimates can be stated in the form of weighted $\setBMO_{\omega}$ (see Section~\ref{sec:pre} for details). These imply, for example, that H\"older continuity can be transferred from $g$ to $\nabla u$ (see Proposition~\ref{cor:hoelder}). This was already proven for all $\frac{2n}{n+2}\leq p$ in \cite{Mis02} and more recently in \cite{KuuMin12} and \cite{Mis13}. However, for the model case  \eqref{eq} and $p\geq2$ considered here, all such estimates are regained by our technique. Moreover, we can weaken the condition on $g$. Indeed, if $g\in L^\infty(I,C^{\gamma(p-1)}(B))$, this already implies that $\nabla u\in C^\gamma_{\text{par}}(I\times B)$ locally for small $\gamma$; see Proposition~\ref{cor:hoelder} at the end of the paper. 

The sub-quadratic case requires more difficult analysis. This can be seen in the elliptic case, where the sub-quadratic case was much more problematic to treat (see \cite{DieKapSch11} for details on that matter). Also in the parabolic case it is not a straightforward extension, but needs other sophisticated tools. We hope to present these in a future work.
Some advances for the $\frac{2n}{n+2}<p<2$ are achieved in this paper. The first important step to gain $\setBMO$ estimates is a decay estimate for homogeneous solutions (called $p$-caloric). In Theorem~\ref{thm:decay} we prove a decay in the spirit of Giaquinta and Modica~\cite{GiaM86} for $p$-caloric solutions. This decay is a distinctively stronger estimate on the H\"older behavior for the gradients of $p$-caloric solutions than known before. It tightens the famous result of DiBenedetto and Friedman \cite{DibFri85} and is therefore of independent interest.

Let us mention some results if the right hand side of \eqref{eq} can be characterized by Radon measures. In case of systems little is known. In the case where $u$ is scalar valued, Kuusi and Mingione provided pointwise estimates, which allow a direct control of $\nabla u$ by the right hand side, such that many regularity properties can be carried over. See \cite{KuuMin122},\cite{KuuMin13}.

Finally we want to give another motivation. In \cite{DieKapSch13} it was possible to extend the techniques of \cite{DieKapSch11} to stationary power law fluids. We hope to gain some generalizations of the estimates given in this article to instationary power law fluids in the future. 

	
 The structure of the paper is as follows: first we prove the decay for $p$-caloric solutions (for all $\frac{2n}{n+2}<p<\infty$). This is done in Section~\ref{sec:decay}. In Section~\ref{sec:bmo} we derive a comparison estimate on so-called intrinsic cylinders (see Lemma~\ref{lem:cubes}). This leads to the boundedness of the intrinsic mean oscillations, which implies the H\"older-Zygmund estimate.

\section{Preliminaries}
\label{sec:pre}
\noindent
Through the paper we will denote by $I$ a (time) interval and $B$ to be a ball (in space). By $I_r,B_r$ we mean a time interval or ball in space with radius $r$. A time space cylinder with ``center point'' $(t,x)$ 
 $Q_{s,r}(t,x):=Q_{s,r}(t,x):=(t,t-s)\times B_r(x)$ and its parabolic boundary as $\partial_{\text{par}}Q_{s,r}(t,x):=[t,t-s]\times \partial B_r(x)\cup (t-s)\times B_r(x)$. As the ``center point`` is mostly of no importance, it will be often omitted.
 We will use the notation
$\mean{f}_E:=\dashint_E f\dx=\frac1{\abs{E}}\int_E f \dx$.  

We have to introduce a few function spaces. Let $\omega:\setR_+\to\setR_+$ almost increasing. This means, that there is a $c>0$ fixed, such that $\omega(r)\leq c\omega(\rho)$ for all $r<\rho$. We say that $f\in\setBMO^{\text{par}}_\omega(Q)$ the weighted space of mean oscillations, if
\[
 \norm{f}_{\setBMO^{\text{par}}_{\omega}(Q)}=\sup\limits_{Q_{r^2,r}\subset Q}
\frac1{\omega(r)}\dashint_{Q_{r^2,r}}\abs{f-\mean{f}_{Q_{r^2,r}}}\dx\dt<\infty.
\]
 For $\omega(r)=1$, we get the space of parabolic bounded mean oscillation: $\setBMO_{\text{par}}(Q)$. By the Campanato characterization, of H\"older spaces we find for $\beta\in (0,1)$ and $\omega(r)=r^\beta$ the space of H\"older continuous function in the parabolic metric.

We will now look at the Bochner spaces of refined BMO.
Let $\omega:\setR_+^2\to\setR_+$. We say that $f\in \setBMO_{\omega}(I\times B)$ if
\begin{align*}
 \norm{f}_{\setBMO_\omega(Q)}:=\sup\limits_{I_s\times B_r\subset Q}\frac1{\omega(s,r)}\dashint_{I_s}\dashint_{B_r}\abs{f-\mean{f(t)}_{B_r}}\dx\dt<\infty.
\end{align*}
 if $\omega\equiv1$, then we have the space $L^\infty(I,\setBMO(B))$. More general, if $\omega$ only depends on $r$, then we have the $L^\infty(I,\setBMO_\omega(B))$ spaces.
 
Through the paper we will need the following typical estimate for mean oscillations, which we will refer to as \textit{best constant property}. For $f\in L^p(Q)$, $p\in[1,\infty)$ we have that
\[
 \dashint_Q\abs{f-\mean{f}_{Q}}^p\dx\leq 2^p\dashint\abs{f-c}^p\dx\text{ for all }c\in\setR.
\]
We will also need the famous \textit{John-Nierenberg estimate}~\cite{JohNir61} 
\begin{align*}
    \dashint_B \abs{f- \mean{f}_B}^q \,dx &\leq c_q\,
    \norm{f}_{\setBMO(B)}^q
  \end{align*}
for $1\leq q<\infty$.
Let us introduce the H\"older--Zygmund spaces. We say that $f\in \mathcal{C}^\gamma(\Omega)$ if 
\[
\norm{f}_{\mathcal{C}^\gamma
(\Omega)}:=\sup\limits_{x\in\Omega}\sup\limits_{[x,x+2h]\subset\Omega}\frac{\abs{f(x+2h)-2f(x+h)+f(x)}}{\abs{h}^\gamma}+\norm{f}_\infty<\infty.
\]
 This is a Banach space. By \cite[Sec. 1.2.2]{Tri92} we find that $C^\gamma(\Omega)=\mathcal{C}^\gamma(\Omega)$ if $\gamma\not\in\setN$ but $C^1(\Omega)\subsetneq\mathcal{C}^1(\Omega)$.  

We find in \cite[Section 1.7.2]{Tri92}, that $\mathcal{C}^1$ has a Campanato space like interpretation. 
Analogous to the spaces of $\setBMO_\omega$ we define the space of weighted bounded linear oscillation $\text{BLO}_{\omega}$ by the semi-norm
\[
 \norm{f}_{\setBLO^q_{\omega}(\Omega)}:=\sup_{B_r\subset \Omega} \inf_{\ell\in P^1(B_r)}\frac1{\omega{(r)}}\bigg(\dashint_{B_r}\Bigabs{\frac{f-\ell}{r}}^q\dx\bigg)^\frac1q,\, 1<q<\infty.
\]
Here $P^1$ is the set of all polynomials with degree 1. For $q=2$ we define $\linear{f}{r}$ as the best linear approximation of $f$ on $B_r$ in with respect to $\norm{\cdot}_2$, which is well defined for all $r>0$ and $f\in L^2_{\text{loc}}$. We find by \cite[Section: 1.7.2]{Tri92} that $\setBLO(\Omega):=\setBLO^1_1(\Omega)\equiv\setBLO^q_1(\Omega)\equiv\mathcal{C}^1(\Omega)$ for all $1\leq q<\infty$; more general, for $\gamma\in(0,1)$ and $\omega(r)=r^\gamma$ the space $\setBLO_{\omega}^q(\Omega)=\mathcal{C}^{1+\gamma}(\Omega)$ for $1\leq q<\infty$. We define that $f$ is in the space of vanishing linear oscillations VLO if $\norm{f}_{\setBLO(B_r(x))}\to 0$ for $r\to 0$ uniform in $x$. 
Please note
\[
 \frac1{\omega(r)}\norm{f}_{\setBMO^q(B_r)}\leq c\norm{f}_{\setBMO_\omega^q(B_r)}\text{ or }\frac1{\omega(r)}\norm{f}_{\text{BLO}^q(B_r)}\leq \norm{f}_{\text{BLO}^q_\omega(B_r)},
\]
because $\omega$ is almost increasing. We will use this in this work without further reference.  

We denote by
\[
 \osc{f}{E}:=\sup\limits_{x,y\in E}\abs{f(x)-f(y)}
\]
the oscillations of $f$ on $E$.

We define the following natural quantity: for $Q\in \RNn$ we have $V(Q):=\abs{Q}^\frac{p-2}2Q$. If $\nabla u\in L^p$, then $V(\nabla u)\in L^2$, therefore $V(\nabla u)$ can be seen as a linear substitute. 
First remark that we will use without further mentioning that for any set $E\subset R^n$ and $f,h\in L^p(E,\RNn)$
\[
 \mean{\abs{f}^p}_E\leq c\dashint_E\abs{V(f)-V(h)}^2\dx+ \mean{\abs{h}^p}_E.
\]
We will need \cite[Lemma 3]{DieE08}. It quantifies the ellipticity of \eqref{eq} in terms of $V$. In our case it states for $P,Q\in \RNn$ and $1<p<\infty$
\begin{align}
 \label{eq:hammer}
\begin{aligned}
(\abs{Q}^{p-2}Q-\abs{P}^{p-2}P)\cdot(Q-P)\sim \abs{V(Q)-V(P)}^2\\
\abs{\abs{Q}^{p-2}Q-\abs{P}^{p-2}P}\sim (\abs{Q}+\abs{Q-P})^{p-2}\abs{P-Q}^2.
\end{aligned}
\end{align}
This implies for $p\geq 2$
\begin{align}
 \label{eq:nervig}
\abs{P-Q}^p\leq c\abs{V(Q)-V(P)}^2.
\end{align}
We also need some estimate which makes use of so called shifted N--functions \cite[Lemma 32]{DieE08} and \cite[(2.5)]{DieKapSch11} we gain for $P,Q,G_1,G_0\in \RNn$ and $\delta>0$
\begin{align}
 \label{eq:hammer2}
\begin{aligned}
&\abs{G_1-G_0}\abs{P-Q}\\
&\quad\leq c(\abs{Q}+\abs{G_1-G_0})^{p'-2}\abs{G_1-G_0}^2+\delta\abs{V(Q)-V(P)}^2.
\end{aligned}
\end{align}
 Here $c$ only depends on $p,n,N$ and $\delta$. We use $p':=\frac p{p-1}$ as the dual exponent to $p$.

 Finally we introduce the $\lambda$--scaled cylinders $Q^\lambda_r(t,x):=(t,t-\lambda^{p-2}r^2)\times B_r(x)$, where $p$ is the exponent of \eqref{eq}.
 For $\theta\in \setR^+$ we define 
$\theta Q_r^\lambda(t,x):=(t,t-\lambda^{2-p}(\theta r)^2)\times B_{\theta r}(x)$. If $\lambda=1$, then we have a standard parabolic cylinder and we write $Q^1_r(t,x)=:Q_r(t,x)$. As solutions are translation invariant and our estimates are local, the center $(t,x)$ of the cube is mostly of no importance and will often be omitted, to shorten notation.  
Finally, we call a cylinder $K$-intrinsic with respect to $f$, when
\begin{align}
 \label{eq:Kint}
\begin{aligned}
\frac\lambda K \leq \mean{\abs{Df}^p}^\frac1p_{Q^\lambda_r}&\leq K\lambda
\text{ and $K$-sub-intrinsic w.r.t $f$, when}\\
\mean{\abs{Df}^p}^\frac1p_{Q^\lambda_r}&\leq K\lambda.
\end{aligned}
\end{align}
We say (sub-)intrinsic if $K=1$.

\section{Decay for p-Caloric Functions}
\label{sec:decay}
\noindent
In this section we consider $h:Q_T\to \setR^N$ to be locally $p$-caloric on a space time domain $Q_T$. I.e. $h$ is a solution to the following system
\[
 \partial_t h-\divergence(\abs{\nabla h}^{p-2}\nabla h)=0
\]
locally in $Q_T$.
In this section we provide a decay for the natural quantity $V(\nabla h)=\abs{\nabla h}^\frac{p-2}{2}\nabla h$. It is an extension to the known result of DiBenedetto and Friedmann \cite{DibFri85} providing finer estimates for the continuity behavior. 
Our results are very much in the spirit of Giaquinta and Modica \cite[Proposition 3.1-3.3]{GiaM86}. We will prove a parabolic version of their decay for the p-caloric setting.

The first theorem we will need is the well-known weak Harnack inequality first proved by DiBenedetto and Friedmann~\cite{DibFri85}, see also \cite[VIII]{DiB93}. We will use the K-sub-intrinsic version
 of \cite[Lemma 1+2]{AceMin07}.
\begin{theorem}
\label{thm:sup}
 Let $p>\frac{2n}{n+2}$ and $h$ be p-caloric on $Q_T$. If for $Q_r^\lambda\subset Q_T$
\[
 \dashint_{Q_r^\lambda}\abs{\nabla h}^p\dz 
\leq K\lambda^p,
\]
then
\[
 \sup_{\frac12Q_r^\lambda}\abs{\nabla h}\leq c\lambda.
\]
The constant only depends on $K,p$ and the dimensions.
\end{theorem}
\begin{proof}
 If $p\geq2$ it is the same statement as in \cite[Lemma 1]{AceMin07}. But also in the case of $\frac{2n}{n+2}<p<2$ the statement holds. In \cite[Lemma 2]{AceMin07} it is proved that if 
\[
 \dashint_{Q_{s^2,\lambda^\frac{p-2}2s}}\abs{\nabla h}^p\dz 
\leq K\lambda^p,
\]
it follows
\[
 \sup_{Q_{s^2,\lambda^\frac{p-2}2s}}\abs{\nabla h}\leq c\lambda. 
\]
Now we define $r=\lambda^\frac{p-2}2s$ which implies, that $s^2=\lambda^{2-p}r^2$. Therefore the estimate holds for all $\frac{2n}{n+2}<p<\infty$. 
\end{proof}
The main theorem of this section is the following.
\begin{theorem}
 \label{thm:decay}
Let $\partial_t h-\divergence(\abs{\nabla h}^{p-2}\nabla h)=0$ on $Q_{\rho}^\lambda$, such that
\[
 \frac{\lambda}{K}\leq\bigg(\dashint_{Q_{\rho}^\lambda}\abs{\nabla h}^p\dz\bigg)^\frac1p \leq K\lambda,
\]
then there exists a $c>0$ and $\alpha,\tau\in (0,1)$ depending only on $n,N,p,K$, such that for every $\theta\in(0,\tau]$
\[
 \sup_{z,w\in \theta Q_{\rho}^\lambda}\abs{V(\nabla h(w))-V(\nabla h(z))}^2
\leq 
c\theta^\alpha\dashint_{Q_{\rho}^\lambda} \abs{V(\nabla h)-\mean{V(\nabla h)}_{Q_{r}^\lambda}}^2\dz.
\]
\end{theorem}
We start with a K-intrinsic cube $Q_\rho^\lambda\subset Q_T$ fixed.
To be able to state the result neatly we define for $r<\rho$
\begin{align}
  M(r)&:=\sup_{Q_r^\lambda}\abs{D h}
  \\
  \Phi(r)&:=\bigg(\dashint_{Q_r^\lambda} 
\Bigabs{V(D h)-\mean{V(\nabla h)}_{Q_r^\lambda}}^2\dz\bigg)^\frac12.
\end{align}

The classic elliptic result of Giaquinta and Modica \cite{GiaM86} was that there is a uniform constant $c$ and an $\alpha\in (0,1)$, such that $\Phi(\theta\rho)\leq c\theta^\alpha\Phi(\rho)$. It is then a standard procedure to gain the estimate of the oscillations. It actually follows by Lemma~\ref{lem:osc2} which can be found in the appendix. 
%
\begin{theorem}
\label{thm:decay1}
 Let $h$ be p-caloric on $Q_{r}^\lambda$, such that
\[
 \bigg(\dashint_{Q_{r}^\lambda}\abs{\nabla h}^p\dz\bigg)^\frac1p\leq K \lambda,
\]
then there exists an $\alpha,c>0$ depending only on $n,N,p,K$, such that for every $\theta\in(0,\frac14]$
\[
 \sup_{z,w\in \theta Q_{r}^\lambda}\abs{V(\nabla h(w))-V(\nabla h(z))}^2
\leq 
c\theta^\alpha \lambda^p.
\]
\end{theorem}
The theorem is a consequence of \cite[IX, Prop 1.1,1.2]{DiB93}, resp.  \cite[Prop. 3.1-3.3]{KuuMin12}. We combine these statements in  the following proposition, as we will use them.
\begin{proposition}
\label{pro:cases}
 Let $h$ be p-caloric. Let
\[
M(\rho)\leq K\lambda.
    \]
Then one of the two alternatives hold:

Case 1, non degenerate: There exist $\beta,\delta_0\in(0,1)$ depending only on $n,N,p,K$ such that 
\begin{align*}
 \frac{\lambda}{4}\leq \inf\limits_{2\delta_0Q^\lambda_r}\abs{\nabla h}&\leq \sup\limits_{2\delta_0Q^\lambda_r}\abs{\nabla h}\leq K\lambda\\
 \text{ and }\osc{V(\nabla h)}{Q^\lambda_{\delta\rho}}^\frac12&\leq c\delta^\beta\Phi(\rho)\text{ for all }\delta\in(0,\delta_0).
\end{align*}

Case 2, degenerate: There exist $\sigma,\eta\in (0,1)$ depending only on $n,N,p,K$ such that
\[
 M(\sigma \rho)\leq \eta K\lambda.
\]
\end{proposition}
%
\begin{proof}
We only have to show that in Case~1,  $\osc{V(\nabla h)}{Q^\lambda_{\delta\rho}}^\frac12\leq c\delta^\beta\Phi(\rho)$ for $\delta\in(0,\delta_0)$. Anything else can be found in \cite[Proposition 3.1-3.3]{KuuMin12}.

  By \cite[Proposition 3.1]{KuuMin12} we know, that if Case~1 does not hold, there exists $\delta_1\in(0,1)$ such that for every sub cube $Q^\lambda_r(z)\subset \delta_1 Q^\lambda_\rho$ we have 
\[
 \frac{\lambda}{4}\leq \inf\limits_{Q^\lambda_r(z)}\abs{\nabla h}\leq \sup\limits_{Q^\lambda_r(z)}\abs{\nabla h}\leq K\lambda.
\]
Therefore we have by \cite[Proposition 3.2]{KuuMin12} for all these sub cubes
\[
 \dashint_{\theta Q^\lambda_r(z)}\abs{V(\nabla h)-\mean{V(\nabla h)}_{\theta Q^\lambda_r(z)}}^2\leq c\theta^{2\beta}\dashint_{Q^\lambda_r(z)}\abs{V(\nabla h)-\mean{V(\nabla h)}_{Q^\lambda_r(z)}}^2.
\]
this implies the result by Lemma~\ref{lem:osc2} with $\delta_0=\frac{\delta_1}{2}$.
\end{proof}

\begin{proof}[Proof of Theorem~\ref{thm:decay}]
Before we can prove the decay we have to do some preliminary work.
If for $Q^\lambda_\rho$ Case~1 of Proposition~\ref{pro:cases} holds, we have the desired decay. 
%

If Case~2 holds, we shall iterate. In this case the degenerate alternative of Proposition~\ref{pro:cases} holds for $Q^\lambda_\rho$. We will now construct another smaller cube on which we can apply Proposition~\ref{pro:cases} again.
 
We find for $\lambda_1=\eta \lambda$, 
\[
Q^{\lambda}_{\sigma\eta^\frac{2-p}2\rho}\subset Q^{\lambda_1}_{\sigma\rho}\subset Q^{\lambda}_{\sigma\rho}\text{ if } p<2,\text{ and }
Q^\lambda_{\sigma\eta^\frac{p-2}2\rho}\subset Q^{\lambda_1}_{\sigma\eta^\frac{p-2}2\rho} \subset Q^{\lambda}_{\sigma\rho} \text{ if $p\geq 2$}.
\] 
We define 
\begin{align*}
\begin{aligned}
 &\rho_1=a\rho\text{ where }\text{$a=\sigma$ for $p<2$ and $a=\sigma\eta^\frac{p-2}2$ for $p\geq 2$}\\
&\text{and }r_1=b\rho\text{ with }b=\eta^\frac{2-p}2\sigma \text{ if }p<2\text{ and }b=\sigma\text{ if } p\geq 2.
\end{aligned}
\end{align*}
We find
\[
 M(r_1)\leq \sup\limits_{Q^{\lambda_1}_{\rho_1}}\abs{\nabla h}\leq M(\sigma\rho)\leq K\eta\lambda=K\lambda_1.
\]
Thus $Q^{\lambda_1}_{\rho_1}$ satisfies the assumption of Proposition~\ref{pro:cases}. If Case~2 holds for this cube we can iterate further with
\begin{align}
 \label{eq:it}
\begin{aligned}
 \lambda_i=\eta^i\lambda;\,\rho_i=a\rho_{i-1}
\text{ and }r_i=br_{i-1},
\end{aligned}
\end{align}
and $a,b$ defined above. 
%
If Case~2 holds also for $Q^{\lambda_{j}}_{\rho_{j}}$ and $1\leq j\leq i-1$, then we find
\[
 Q^{\lambda}_{r_i}\subset Q^{\lambda_i}_{\rho_i}\subset Q^{\lambda_{i-1}}_{\rho_{i-1}}\text{ and }\sup\limits_{Q^{\lambda_i}_{\rho_i}}\abs{\nabla h}\leq\sup\limits_{Q^{\lambda_{i-1}}_{\sigma\rho_{i-1}}}\abs{\nabla h} \leq K\eta\lambda_{i-1}=K\eta^i\lambda.
\]
Let us fix $m\in \setN$, such that $\eta^mK^2\leq\frac12$. This implies that if the degenerate alternative holds for all $i\leq m$, then
\begin{align}
\label{eq:deg}
\sup_{Q_{r^m}^{\lambda}}\abs{\nabla h}\leq \sup_{Q^{\lambda_m}_{\rho_m}}\abs{\nabla h}\leq K\eta^m\lambda\leq 
\frac 12\mean{\abs{\nabla h}^p}_{Q^\lambda_\rho}^\frac1p
\end{align}
by the assumption that $Q^\lambda_\rho$ is intrinsic.

Now we are able to prove the decay.
Let us first assume, that for one $i\in\set{0,...,m}$ the non-degenerate Case~1 of Proposition~\ref{pro:cases} holds. This implies for $\delta\in (0,\tau)$, where $\tau=\frac{\delta_0}{b^m}$, that
\[
 \osc{V(\nabla h)}{\delta Q_{\rho}^\lambda}^\frac12\leq \osc{V(\nabla h)}{\delta b^m Q_{r_m}^\lambda}^\frac12\leq c\delta^\beta\Phi^{\lambda_i}(\rho_i)\leq c\delta^\beta\Phi^{\lambda}(\rho),
\]
as
\[
Q^{\lambda_i}_{\rho_i}\subset Q^\lambda_\rho\text{ and }\frac{\abs{Q^\lambda_\rho}}{\abs{Q^{\lambda_i}_{\rho_i}}}\leq \frac{\abs{Q^\lambda_\rho}}{\abs{Q^{\lambda_m}_{\rho_m}}}\leq c\text{ depending only on $n,N,p,K$.}
\]
This leaves the case, when for all $i\in\set{0,...,m}$ the degenerate alternative (Case~2) holds. In this case we know by \eqref{eq:deg}
\[
\sup_{Q^{\lambda_m}_{\rho_m}}\abs{\nabla h}\leq K\eta^m\lambda\leq 
\frac 12\mean{\abs{\nabla h}^p}_{Q^\lambda_\rho}^\frac1p.
\]
 This implies that
\[
  \abs{\mean{V(\nabla h)}_{Q^{\lambda_m}_{\rho_m}}}\leq\frac1{2^\frac p2}\mean{\abs{V(\nabla h)}^2}_{Q^\lambda_\rho}^\frac12.
\]
Therefore we gain by Lemma~\ref{lem:osc}.
\[
 \lambda^p\leq K^p\mean{\abs{\nabla h}^p}_{Q^\lambda_\rho}\leq c\dashint_{Q_{\rho}^\lambda} \abs{V(\nabla h)-\mean{V(\nabla h)}_{Q_{r}^\lambda}}^2\dz,
\]
again, as 
\[
Q^{\lambda_i}_{\rho_i}\subset Q^\lambda_\rho\text{ and }\frac{\abs{Q^\lambda_\rho}}{\abs{Q^{\lambda_i}_{\rho_i}}}\leq \frac{\abs{Q^\lambda_\rho}}{\abs{Q^{\lambda_m}_{\rho_m}}}\leq c\text{ depending only on $n,N,p,K$.}
\]
Finally, the last estimate combined with Theorem~\ref{thm:decay1} implies the decay also in this case.
\end{proof}

\section{A BMO result for $p\geq 2$}
\noindent
Theorem~\ref{thm:m} is a consequence of a more general result. From this we will conclude also other Campanto like estimates. 

Before proving the main result we will have to prove some intermediate results. The key ingredient is to carefully choose a family of intrinsic cylinders. 
\subsection{Finding a scaled sequence of cubes}
To treat the scaling behavior in a way to gain a $\setBMO$ result for $\eqref{eq}$ is quit delicate. Our estimates are based on comparison principles: Whenever one knows that $\normg$ is small, then $u$ is ''close'' to a p-caloric comparison solution.
%

In the following we will construct sub-intrinsic cubes
with properties convenient for our needs.
\begin{lemma}
\label{lem:scal}
Let $p\geq2$. Let $Q_{S,R}(t,x)\subset Q_T$ and $b\in (0,2)$. For every $0<r\leq R$ there exists $s(r)$, $\lambda_r$ and $Q_{s(r),r}(t,x)$ with the following properties. Let $r,\rho\in (0,R]$ and $r<\rho$, then 
\begin{enumerate}
\item\label{scal:-1}  $0\leq s(r)\leq S$ and $s(r)=\lambda_r^{2-p}r^2$. Especially $Q_{s(r),r}(t,x)=Q_r^{\lambda_r}\subset Q_T$.
\item\label{scal:0}  $s(r)\leq \big(\frac{r}{\rho}\big)^bs(\rho)$, the function $s$ is continuous and strictly increasing on $[0,R]$. Especially $Q_r^{\lambda_r}\subset Q_\rho^{\lambda_{\rho}}$.
 \item \label{scal:2} $\dashint_{Q_r^{\lambda_r}}\abs{\nabla u}^p\dz\leq\lambda_r^p$, i.e. $Q_r^{\lambda_r}$ is sub-intrinsic.
\item \label{scal:5}  if $s(r)<\big(\frac{r}{\rho}\big)^bs(\rho)$, then there exists $r_1\in[r,\rho)$ such that $Q^{\lambda_{r_1}}_{r_1}$ is intrinsic.
\item \label{scal:6} if for all $r\in (r_1,\rho)$, $\Qr{}$ is strictly sub-intrinsic, then $\lambda_{r}\leq\big(\frac{r}{\rho}\big)^\beta\lambda_{\rho}$ for all $r\in [r_1,\rho]$ and $\beta=\frac{2-b}{p-2}\in (0,\frac{2}{p-2})$.
 \item \label{scal:3} for $\theta\in (0,1]$, $\theta^\beta\lambda_{r}\leq\lambda_{\theta r}\leq \frac{c\lambda_r}{\theta^\frac{n+2}2}$.
\item \label{scal:4} for $\theta\in(0,1], \abs{Q_{\theta r}^{\lambda_{\theta r}}}^{-1}
\leq c\theta^{-(n+2)(1+\frac{p-2}2)}\abs{{Q_{r}^{\lambda_{r}}}}^{-1} $.
\item \label{scal:7} for $\theta\in(0,1]$, we find $\Qr{\sigma}\subset\theta\Qr{}$ for $\sigma=\theta^\frac{2}{b}$.
\end{enumerate}
The constant only depends on the dimensions and $p$.
\end{lemma}
\begin{proof}
Let $Q_{S,R}(t,x)\subset Q_T$. In the following we often omit the point $(t,x)$.
We start, by defining for every $r\in (0,R]$ 
\begin{align}
\label{eq:s}
 &\tilde{s}(r)=\max\Bigset{s\leq S\Big|\bigg(\int_{t-s}^{t}
  \int_{B_r(x)}\abs{\nabla u}^p \dz\bigg)^{p-2}s^2\leq r^{2p}\abs{B_r}^{p-2}}.
\end{align}
The function $\tilde{s}(r)$ is well defined and strictly positive for $r>0$.
We define $\tilde{\lambda}_r$ by the equation $r^2\tilde{\lambda}_r^{2-p}=\tilde{s}(r)$. We will first show, that $Q^{\tilde{\lambda}_r}_r:=Q_{\tilde{s}(r),r}$ holds \ref{scal:2}. By construction we find, that
\begin{align}
  \label{eq:sforp}
\begin{aligned}
  \bigg(\int_{Q_{r,\tilde{s}(r)}}\abs{\nabla u}^p \dz\bigg)^{p-2}\tilde{s}(r)^2\leq r^{2p}\abs{B_r}^{p-2}.
\end{aligned}
\end{align}
This implies that
 \[
 \bigg(\dashint_{Q_{r,\tilde{s}(r)}}\abs{\nabla u}^p\dz\bigg)^{p-2} \tilde{s}(r)^p\leq r^{2p}=(\tilde{\lambda}^{(p-2)}\tilde{s}(r))^p
\]
which implies
\begin{align}
\dashint_{Q_{r,\tilde{s}(r)}}\abs{\nabla u}^p\dz \leq \tilde{\lambda}_r^p,
\label{eq:sleqS}
\text{ and if }
\dashint_{Q_{\tilde{s}(r),r}}\abs{\nabla u}^p\dz<\tilde{\lambda}_r^p\text{, then $\tilde{s}(r)=S$.}
\end{align}
Next we will show, that $\tilde{s}(r)$ is continuous for $r\in(0,R]$.
For $\epsilon \leq \tilde{s}(r)\leq S-\epsilon$ and $r_0>0$, we find that $\big(\int^t_{t-\tilde{s}(r)}\int_{B_r}\abs{\nabla u}^p \dz\big)^{p-2} s^2$ is growing of order 2.
Because the growth rate is explicitly bounded by  
\[
\frac{\abs{B_R}^{p-2}R^{2p}}{\epsilon^2}\geq \bigg(\int^t_{t-\tilde{s}(r)}\int_{B_r(x)}\abs{\nabla u}^p \dz\bigg)^{p-2}\geq \frac{r_0^{2p}\abs{B_{r_0}}^{p-2}}{S^2},
\]
for $r\in[r_0,R]$.
This implies that there exists a $\delta_{\epsilon,r_0} >0$, such that for all $r,r_1\in [r_0,R]$ with $\abs{r-r_1}<\delta_{\epsilon,r_0}$
\begin{align*}
 &\bigg(\int_{t-\tilde{s}(r)}^{t}
  \int_{B_{r_1}(x)}\abs{\nabla u}^p \dz\bigg)^{p-2}(\tilde{s}(r)-\epsilon)^2 < r_1^{2p}\abs{B_{r_1}^{p-2}} \\
&\qquad<\bigg(\int_{t-\tilde{s}(r)}^{t}
  \int_{B_{r_1}(x)}\abs{\nabla u}^p \dz\bigg)^{p-2}(\tilde{s}(r)+\epsilon)^2.
\end{align*}
as $\big(\int_{t-s}^{t}\int_{B_r(x)}\abs{\nabla u}^p \dz\big)^{p-2}$ and $r^{2p}\abs{B_r}^{p-2}$ are both uniformly continuous in $r$.
Now we gain immediately
\begin{align*}
 &\bigg(\int_{t-\tilde{s}(r)+\epsilon}^{t}
  \int_{B_{r_1}(x)}\abs{\nabla u}^p \dz\bigg)^{p-2}(\tilde{s}(r)-\epsilon)^2 < r_1^{2p}\abs{B_{r_1}^{p-2}} \\
&\qquad<\bigg(\int_{t-\tilde{s}(r)-\epsilon}^{t}
  \int_{B_{r_1}(x)}\abs{\nabla u}^p \dz\bigg)^{p-2}(\tilde{s}(r)+\epsilon)^2,
\end{align*}
which implies that $\abs{\tilde{s}(r)-\tilde{s}(r_1)}<2\epsilon$.
 
Let us define $s_\epsilon(r)=\max\set{\epsilon,\min\set{\tilde{s}(r), S-\epsilon}}$. By the previous calculations we find that $s_{\epsilon}$ is uniformly continuous, especially $\abs{s_\epsilon(r)-s_\epsilon(r_1)}\leq 2\epsilon$ for $r,r_1\in [r_0,R]$ with $\abs{r-r_1}<\delta_{\epsilon,r_0}$. Therefore 
\[
 \abs{\tilde{s}(r_1)-\tilde{s}(r)}\leq \abs{\tilde{s}(r_1)-s_\epsilon(r_1)}+\abs{s_\epsilon(r_1)-s_\epsilon(r)}+\abs{s_\epsilon(r)-\tilde{s}(r)}\leq 4\epsilon.
\]
 As $r_0$ was arbitrary we find that $\tilde{s}(r)$ is continuous on $(0,R]$.

Now it might happen, that $r<\rho$ and $\tilde{s}(r)>\tilde{s}(r)$. To avoid that we define for $b\in (0,2)$
\begin{align*}
s(r)=\min_{R\geq a\geq r}\Big(\frac{r}{a}\Big)^b\tilde{s}(a). 
\end{align*}
The minimum exists, as $\Big(\frac{r}{a}\Big)^b\tilde{s}(a)$ is continuous in $a$. 
As for $\rho\in(r,R]$
\begin{align}
\label{eq:grow}
 s(r)=\min\Bigset{\min_{\rho\geq a\geq r}\Big(\frac{r}{a}\Big)^b\tilde{s}(a),\Big(\frac{r}{\rho}\Big)^bs(\rho)}
\end{align}
we find that
$s(r)< s(\rho)$. Now we define $\lambda_r:=\big(\frac{r^2}{s(r)}\big)^\frac1{p-2}\geq \tilde{\lambda}_r$ and $Q_r^{\lambda_r}:=Q_{s(r),r}$. By this definition we find \ref{scal:-1} and \ref{scal:0}, as $\lim_{r\to0}s(r)\leq \lim_{r\to 0}\big(\frac{r}{R}\big)^bS(R)=0$. 

We show \ref{scal:2}, by \eqref{eq:sforp} 
\begin{align}
\label{eq:mon}
 \dashint_{Q_{s(r),r}}\abs{\nabla u}^p \leq \frac{\tilde{s}(r)}{s(r)} \dashint_{Q_{\tilde{s}(r),r}}\abs{\nabla u}^p
=
\Big(\frac{{\lambda}_r}{\tilde{\lambda}_r}\Big)^{p-2}\dashint_{Q_{\tilde{s}(r),r}}\abs{\nabla u}^p \leq \tilde{\lambda}_r^2{\lambda}_r^{p-2}\leq \lambda_{r}^p.
\end{align}
To prove \ref{scal:5} we assume that $s(r)<\big(\frac{r}{\rho}\big)^bs(\rho)$. Then there exist a $r_1\in [r,\rho)$, such that 
\[
\Big(\frac{r}{r_1}\Big)^b\tilde{s}(r_1)=s(r)=\min_{R\geq a\geq r}\Big(\frac{r}{a}\Big)^b\tilde{s}(a)\leq \Big(\frac{r}{r_1}\Big)^b\min_{R\geq a\geq r_1}\Big(\frac{r_1}{a}\Big)^b\tilde{s}(a)=\Big(\frac{r}{r_1}\Big)^b s(r_1).
\]
Now because $\tilde{s}(r_1)\geq s(r_1)$ we find $\tilde{s}{r_1}=s(r_1)$. Since also $s(r)<\big(\frac{r_1}{\rho}\big)^bs(\rho)\leq \big(\frac{r_1}{R}\big)^b S$ we find by \eqref{eq:sforp} that $Q_{s(r_1),r_1}=Q_{r_1}^{\lambda_{r_1}}$ is intrinsic. This implies \ref{scal:5}. 

To prove \ref{scal:6} we gain by \ref{scal:5} that if $Q_a^{\lambda_a}$ is strictly sub-intrinsic for all $a\in(r,\rho)$, then 
$s(a)=\big(\frac{a}{\rho}\big)^bs(\rho)$ for all $a\in(r,\rho)$.
Now we calculate
\[
 \lambda_a^{p-2}= \frac{a^2}{s(a)}=\frac{a^2}{\big(\frac{a}{\rho}\big)^bs(\rho)}=\Big(\frac{a}{\rho}\Big)^{2-b}\lambda_\rho^{p-2},
\]
this proves \ref{scal:6}, with $\beta=\frac{2-b}{p-2}$.

To prove \ref{scal:3} we take $\theta\in(0,1)$. If $s(\theta r)=\theta^bs(r)$ we are finished. If $s(\theta r)<\theta^bs(r)$, we find by \ref{scal:5} that there is a $\sigma\in[\theta,1)$ with $s(\theta r)=\big(\frac{\theta}{\sigma}\big)^b s(\sigma r)$ and $Q_{\sigma r}^{\lambda_{\sigma r}}$ is intrinsic. This implies using also \ref{scal:2}
\begin{align*}
 \lambda_{\sigma r}^2=\frac{c}{(\sigma r)^{n+2}}\int\limits_{Q_{s(\sigma r),\sigma r}}\abs{\nabla u}^p\dz\leq
 \frac{cs(r)}{r^2\sigma^{n+2}} \dashint_{s(r),r}\abs{\nabla u}^p\dz
 \leq \frac{c\lambda_r^2}{\theta^{n+2}}.            
\end{align*}
By the definition of $\lambda_r$ we find for $\beta=\frac{2-b}{p-2}$ and the previous that
\[
 \lambda_r\leq \theta^{-\beta}\lambda_{\theta r}\text{ and }\lambda_{\theta_r}\leq \frac c{\theta^\frac{n+2}2}\lambda_r,
\]
which implies \ref{scal:3} and \ref{scal:4}.
To prove \ref{scal:7} we take $\theta Q_{s(r),r}=Q_{\theta^2 s(r),\theta r}$ we define $\sigma<\theta$, such that $\sigma^b =\theta^2$. Now we find by \eqref{eq:grow}, that  $s(\sigma r)\leq \sigma^bs(r)=\theta^2 s(r)$. 
\end{proof}

\subsection{Comparison}
In this section we will derive a comparison estimate which will allow us to gain BMO estimates.
Let $u$ be a solution to \eqref{eq} on $I\times B$. As we want to use Theorem~\ref{thm:decay}, we will have to start with an intrinsic cylinder. We therefore take any intrinsic cylinder $Q_R^{\lambda_0}(z)\subset I\times B$, i.e.
\[
 \dashint_{Q_R^{\lambda_0}(z)}\abs{\nabla u}^p =\lambda_0^p.
\]
In this section we define $Q^{\lambda_r}_r$ as the sub-intrinsic cylinders all sharing the same center, which are constructed by Lemma~\ref{lem:scal}.
 By comparison we mean the local comparison to a p-caloric function. I.e.\ for $r\in (0,R)$ we will compare $u$ to solutions of
\begin{align}
\label{eq:hom}
\begin{aligned}
 \partial_t h-\divergence(\abs{\nabla h}^{p-2}\nabla h)&=0\text{ on }Q_{r}^{\lambda_r}\\
h&=u\text{ on }\partial_{par}Q_{r}^{\lambda_r}.
\end{aligned}
\end{align}
\begin{lemma}
 \label{eq:comparison}
Let $p\geq 2$, $(t,t-\lambda_r^{2-p})\times B_r(x)=:\Qr{}\subset I\times B$ and $g\in L^\infty(I,\setBMO(B))$.  For $h$ the solution of \eqref{eq:hom} and $u$ the solution of \eqref{eq} we have
\begin{align*}
\begin{aligned}
\lambda_r^{p-2}\dashint_{B_r(x)}&\frac{\abs{u-h}^2(t)}{r^2}\dy +
 \dashint_{\Qr{}}\abs{V(\nabla u)-V(\nabla h)}^2\dz
\leq  c\norm{g}_{L^\infty(I, \setBMO(B_r(x)))}^{p'}.
\end{aligned}
\end{align*}
\end{lemma}
\begin{proof}
 We take $u-h$ as a test function for both systems \eqref{eq} and \eqref{eq:hom}. We take the difference and find
\begin{align*}
 &\dashint_\Qr{}\partial_t\frac{\abs{u-h}^2}{2}\dz+\dashint_\Qr{}(\abs{\nabla u}^{p-2}\nabla u-\abs{\nabla h}^{p-2}\nabla h)\cdot \nabla(u-h)\dz\\
\quad&=
\dashint_\Qr{}g\cdot \nabla(u-h)\dy\dtau=\dashint_\Qr{}(g-\mean{g(\tau)}_{B_r})\cdot \nabla(u-h)\dy\dtau.
\end{align*}
We find by \eqref{eq:hammer}, \eqref{eq:hammer2} and as $p'\leq 2$
\begin{align*}
\begin{aligned}
\lambda_r^{p-2}\dashint_{B_r}&\frac{\abs{u-h}^2(t)}{r^2}\dy +
\dashint_{\Qr{}}\abs{V(\nabla u)-V(\nabla h)}^2\dz\\
&\leq c\dashint_{\Qr{}}\Big(\abs{\nabla u}+\abs{g-\mean{g(\tau)}_{B_r(x)}}\Big)^{p'-2}\abs{g-\mean{g(\tau)}_{B_r(x)}}^2\dy\dtau\\
&\quad+\delta \dashint_{\Qr{}}\abs{V(\nabla u)-V(\nabla h)}^2\dz\\
&\leq c\dashint_{t-\lambda_r^{2-p}r^2}^t\dashint_{B_r(x)}\abs{g-\mean{g(\tau)}_{B_r(x)}}^{p'}\dy\dtau
+\delta \dashint_{\Qr{}}\abs{V(\nabla u)-V(\nabla h)}^2\dz.
\end{aligned}
\end{align*}
We absorb and use John-Nirenberg to find that
\[
 \dashint_{B_r(x)}\abs{g-\mean{g(\tau)}_{B_r}}^{p'}\dx\leq c\norm{g(\tau)}_{\setBMO(B_r(x))}^{p'},
\]
 which leads to the result.
\end{proof}

\begin{proposition}
 \label{pro:comp}
Let $Q^{\lambda_0}_R$ be intrinsic and $r\in(0,R)$ and $g\in L^\infty(I,\setBMO(B))$. Let $\beta\leq\frac{\alpha}{1+\alpha\frac{p-2}2}$, such that $\beta<\frac{2}{p-2}$, where $\alpha$ is defined by Theorem~\ref{thm:decay1}. 
Then there exist $K,c>1$ depending only on $n,N,p,\beta$, such that one of the following two alternatives holds:

Case 1: $\lambda_r^p\leq K\normg^{p'}$

Case 2: For the p-caloric comparison  function $h$ of \eqref{eq:hom} there exist a $\rho\in [r,R]$ such that
\begin{align*}
   \osc{V(\nabla h)}{\Qr{\sigma}}^2\leq c&\Big(\frac{\sigma r}{\rho}\Big)^\beta\dashint_{Q_{\rho}^{\lambda_\rho}}
\abs{V(\nabla u)-\mean{V(\nabla u)}_{Q_{\rho}^{\lambda_{\rho}}}}^2\dz\\
&+c\sigma^\beta\normg^{p'}
\end{align*}
for every $\sigma\in (0,\delta]$ and $Q^{\lambda_{ r}}_{r}$ defined by Lemma~\ref{lem:scal}. The constant $\delta\in (0,1)$ only depends on $n,N,p$. 
\end{proposition}
   \begin{proof}
Suppose Case~1 does not hold.  
We find for $\epsilon=\frac{1}{K}$ 
\begin{align}
\label{eq:gsup}
\normg^{p'} \leq  \epsilon\lambda_r^p.
\end{align}
Now let $h$ be the solution of \eqref{eq:hom} on $\Qr{}$, then
Lemma~\ref{eq:comparison} implies
\begin{align}
\label{eq:hsup}
 \dashint_{\Qr{}}\abs{\nabla h }^p\dz\leq 2^p\dashint_{\Qr{}}\abs{\nabla u}^p\dz+c\normg^{p'}\leq c\lambda_r^p.
\end{align}
We therefore can apply Theorem~\ref{thm:decay1} and find for $\theta<\frac14$
\begin{align}
\label{eq:decay2}
 \osc{V(\nabla h)}{\theta\Qr{}}^2\leq c\theta^\alpha\lambda_r^p.
\end{align}
We define 
\begin{align}
\label{eq:rho}
 \rho:=\min\set{a\geq r|Q_a^{\lambda_{a}}\text{ is }\text{intrinsic}}.
\end{align}
By construction $\rho\leq R$ exists as  $Q^{\lambda_0}_R$ is intrinsic. Moreover, (see Lemma~\ref{lem:scal},\ref{scal:6}), we find that $\lambda_a\leq(\frac{a}{\rho})^\beta\lambda_\rho$ for every $r\leq a\leq \rho$. 

If $\frac{\rho}2>r$, we find
\[
 \mean{\abs{\nabla u}^p}_{\Qrho{\frac12}}\leq \lambda_{\frac12\rho}^p\leq \frac1{2^\beta}\mean{\abs{\nabla u}^p}_{\Qrho{}}.
\]
 Therefore Lemma~\ref{lem:osc} implies
\begin{align}
\label{eq:decay1}
 \lambda_r^p\leq c\Big(\frac{r}{\rho}\Big)^\beta\lambda_\rho^p\leq c\Big(\frac{r}{\rho}\Big)^\beta \dashint_{Q_{\rho}^{\lambda_\rho}}
\abs{V(\nabla u)-\mean{V(\nabla u)}_{Q_{\rho}^{\lambda_{\rho}}}}^2\dz.
\end{align}
If $ \frac\rho2\leq r\leq \rho$, we either find that
\[
 \mean{\abs{\nabla u}^p}_{\Qr{}}\leq \frac 12 \mean{\abs{\nabla u}^p}_{\Qrho{}}
\]
in which case we have \eqref{eq:decay1} again by Lemma~\ref{lem:osc}. Otherwise we have
\[
 \mean{\abs{\nabla u}^p}_{\Qr{}}> \frac 12 \mean{\abs{\nabla u}^p}_{\Qrho{}}\geq  c\lambda_r^p.
\]
We find by Lemma~\ref{eq:comparison} and as Case~1 does not hold
\begin{align}
\label{eq:uequivh}
\begin{aligned}
 \lambda_r^p\leq c\dashint_{\Qr{}}\abs{\nabla u}^p&\leq c\dashint_{\Qr{}}\abs{\nabla h }^p +c\normg^{p'}\\
  &\leq c\dashint_{ \Qr{}}\abs{\nabla h}^p
    +{c\epsilon} \lambda_r^p.
  \end{aligned}
\end{align}
We gain (if $\epsilon$ is small enough) by the previous combined with \eqref{eq:hsup} 
\begin{align}
\label{eq:hlow}
\lambda_r^p\sim\dashint_{\Qr{}}\abs{\nabla h }^p\dz.
\end{align}
Now we can apply Theorem~\ref{thm:decay}. This implies together with Lemma~\ref{eq:comparison} for $\theta\in (0,\tau)$
\[
 \osc{V(\nabla h)}{\theta\Qr{}}^2\leq c\theta^\alpha\!\! \dashint_{Q_{r}^{\lambda_r}}
\abs{V(\nabla u)-\mean{V(\nabla u)}_{\Qr{}}}^2\dz+c\theta^\alpha\normg^{p'}.
\]
Combining the last estimate with \eqref{eq:decay2} and \eqref{eq:decay1} we find
\begin{align*}
 \osc{V(\nabla h)}{\theta\Qr{}}^2\leq &c\theta^\alpha \Big(\frac{r}{\rho}\Big)^\beta \dashint_{\Qrho{}}
\abs{V(\nabla u)-\mean{V(\nabla u)}_{\Qrho{}}}^2\dz\\
&+c\theta^\alpha\normg^{p'}.
\end{align*}
To conclude the proof we use Lemma~\ref{lem:scal}, \ref{scal:7}: For $\sigma^\frac{b}{2}=\theta$ we have $\Qr{\sigma}\subset \theta \Qr{}$, therefore
 \begin{align*}
 &\osc{V(\nabla h)}{\Qr{\sigma}}^2\\
&\quad\leq c\sigma^\frac{b\alpha}{2} \Big(\frac{r}{\rho}\Big)^\beta \dashint_{\Qrho{}}
\abs{V(\nabla u)-\mean{V(\nabla u)}_{\Qrho{}}}^2\dz+c\sigma^\frac{b\alpha}{2} \normg^{p'}\\
&\quad\leq c\Big(\frac{\sigma r}{\rho}\Big)^\beta \dashint_{\Qrho{}}
\abs{V(\nabla u)-\mean{V(\nabla u)}_{\Qrho{}}}^2\dz+c\sigma^\beta\normg^{p'}
\end{align*}
by the choice of $\beta=\frac{2-b}{p-2}\leq\frac{b \alpha }{2}$ by our assumptions on $\beta$.
   \end{proof}

\subsection{An intrinsic BMO result}
\label{sec:bmo}

The next proposition gives an intrinsic BMO estimate. We will prove it for the refined spaces $\setBMO_\omega$. In the following let $\omega:[0,\infty)\to [0,\infty)$ be almost increasing. Moreover, 

\begin{align}
 \label{eq:omega}
\frac{\omega(r)}{\omega(\sigma r)}\leq  c_1\sigma^\frac{-\gamma}{p}\text{ for }\sigma\in(0,1] \text{ where }
\gamma<\min\Bigset{\frac{\alpha}{1+\alpha\frac{p-2}2},\frac2{p-2}}.
\end{align}

\begin{lemma}
\label{pro:int}
Let $Q^{\lambda_0}_R$ be intrinsic, $\omega$ hold \eqref{eq:omega} and $g\in L^\infty(I,\setBMO_{\omega'}(B))$, with $\omega'\equiv\omega^{p-1}$. Then there exist constants $c,\beta$ depending on $\gamma, c_1,n,N,p$ such that
\begin{align*}
\sup\limits_{0<r<R}\frac{1}{\omega^p(r)}\dashint_{ \Qr{}}\abs{V(\nabla u)-\mean{V(\nabla u)}_{\Qr{}}}^2
&\leq
 c\normgw{r}^{p'}\\
&+ \frac{c}{\omega^p(R)}\dashint_{ Q^{\lambda_{ 0}}_{ R}}\abs{V(\nabla u)-\mean{V(\nabla u)}_{Q^{\lambda_{ 0}}_{ R}}}^2,
\end{align*}
where $\Qr{}$ is defined by Lemma~\ref{lem:scal} for a $\beta>\gamma$ fixed.
\end{lemma}
\begin{proof}
We fix $\gamma<\beta<\min\Bigset{\frac{\alpha}{1+\alpha\frac{p-2}2},\frac2{p-2}}$. Now we take $\sigma\in (0,1)$. We will define the size of $\sigma$ in the end of the proof. If $r\geq\sigma R$, we find by Lemma~\ref{lem:scal}, \ref{scal:4} 
\begin{align}
\label{bmo1}
  \frac 1{\omega^p(r)}\dashint_{ Q^{\lambda_{ r}}_{ r}}\abs{V(\nabla u)-\mean{V(\nabla u)}_{Q^{\lambda_{ r}}_{ r}}}^2\leq \frac {c(\sigma)}{\omega^p(R)}\dashint_{ Q^{\lambda_{ 0}}_{ R}}\abs{V(\nabla u)-\mean{V(\nabla u)}_{Q^{\lambda_{ 0}}_{ R}}}^2.
\end{align}
Now we will prove the estimate for $\sigma r\in(0,\sigma R]$. We apply Proposition~\ref{pro:comp} on the cylinder $\Qr{}$. If Case 1 holds, we find as $Q_{\sigma r}^{\lambda_{\sigma r}}$ is sub-intrinsic
\begin{align}
\label{bmo2}
\begin{aligned}
 \frac1{\omega^p(\sigma r)}\dashint_{Q_{\sigma r}^{\lambda_{\sigma r}}}&\abs{V(\nabla u)-\mean{V(\nabla u)}_{Q_{\sigma r}^{\lambda_{\sigma r}}}}^2\dz
\leq \frac{c(\sigma)}{\omega^p(r)}\lambda_r^p\\
&\leq  \frac{cK^p}{\omega^p(r)}\normg^{p'}
\leq c\normgw{r}^{p'}
\end{aligned}
\end{align}
where we used that $\omega$ is almost increasing and that $\omega'\equiv\omega^\frac{p}{p'}$.

If Case~2 of Proposition~\ref{pro:comp} holds, we find using the best constant property, Lemma~\ref{eq:comparison}, \eqref{eq:omega} and Lemma~\ref{lem:scal}~\ref{scal:4}
\begin{align}
\label{bmo3}
\begin{aligned}
& \frac{1}{\omega^p(\sigma r)}
\dashint_{\Qr{\sigma}}
\abs{V(\nabla u)-\mean{V(\nabla u)}}_{\Qr{\sigma}}^2\dz\\
&\quad\leq \frac{c}{\omega^p(\sigma r)}\dashint_{\Qr{\sigma}}
\abs{V(\nabla h)-\mean{V(\nabla h)}_{\Qr{\sigma}}}^2\dz
+\frac{c(\sigma)}{\omega^p(r)} \dashint_{\Qr{}}\abs{V(\nabla u)-V(\nabla h)}^2\dz\\
&\quad\leq \frac c{\omega^p(\sigma r)}\osc{V(\nabla h)}{\Qr{\sigma}}^2
+c\normgw{r}^{p'}.
\end{aligned}
\end{align}
By Proposition~\ref{pro:comp} and \eqref{eq:omega} we find for $\sigma\in (0, \delta)$ and $\rho\geq r$
\begin{align*}
 \frac 1{\omega^p(\sigma r)}\osc{V(\nabla h)}{\Qr{\sigma}}
&\leq \sigma^{\beta-\gamma} \frac c{\omega^p(\rho)} \dashint_{\Qrho{}}\abs{V(\nabla u)-\mean{V(\nabla u)}_{\Qrho{}}}^2\dz\\
&\quad + \sigma^{\beta-\gamma} \frac c{\omega^p(r)}\normg^{p'}
\end{align*}
Combining the last estimate with \eqref{bmo1},\eqref{bmo2} and \eqref{bmo3} leads to
\begin{align}
\label{eq:final}\begin{aligned}
\sup\limits_{a<r<R}\frac{1}{\omega(r)}
&\dashint_{ \Qr{}}\abs{V(\nabla u)-\mean{V(\nabla u)}_{\Qr{}}}^2\\
&\leq
 c\normgw{r}^{p'}
+ \frac{c}{\omega(R)}\dashint_{ Q^{\lambda_{ 0}}_{ R}}\abs{V(\nabla u)-\mean{V(\nabla u)}_{Q^{\lambda_{ 0}}_{ R}}}^2\\
&+c\sigma^{\beta-\gamma} \sup\limits_{\frac{a}\sigma<r<R}
\frac{1}{\omega(r)}\dashint_{ \Qr{}}\abs{V(\nabla u)-\mean{V(\nabla u)}_{\Qr{}}}^2.
\end{aligned}
\end{align}
Now fix $\sigma$ conveniently, such that we can absorb the last term. The result follows by $a\to 0$.
\end{proof}

%
%
 In Proposition~\ref{pro:intrbmo} we show the intrinsic BMO estimate. Before we need another lemma on cylinders.

\begin{lemma}
 \label{lem:cubes}
Let $Q^{\lambda_0}_R$ be sub-intrinsic. For every $z\in Q^{\lambda_0}_{R/2}$ there exist a sub-intrinsic cube $Q^{\lambda_{R/2}}_{R/2}(z)\subset Q_R^{\lambda_0}$ and $\lambda_{R/2}\sim\lambda_0$.

Let $Q_R=(t,t-R^2)\times B_R(x)$. Then for every $z\in Q_{R/2}$ there exists a sub-intrinsic cube $Q^{\lambda_{R/2}}_{R/2}(z)\subset Q_R$ and $\lambda_{R/2}\sim\max\bigset{\big(\dashint_{Q_R}\abs{\nabla u}^p\big)^\frac12,1}$.
\end{lemma}
\begin{proof}
We start with the first statement. Since $Q^{\lambda_0}_R$ is sub-intrinsic we find for fixed $z\in Q_\frac{R}2^{\lambda_0}$
\[
\frac{1}{\abs{Q^{\lambda_0}_R}}\int_{Q_{\frac R2}^{\lambda_0}(z)}\abs{\nabla u}^p\leq \lambda_0^p.
\]
Hence, for  $2^\frac{n+2}{p-2}\lambda_0=\lambda_{R/2}\geq\lambda_0$ we find
\[
 \bigg(\dashint_{Q_{\frac R2}^{\lambda_\frac{R}2}(z)}\abs{\nabla u}^p\bigg)^\frac1p\leq \lambda_{R/2}\leq 2^\frac{n+2}{p-2}\lambda_0.
\]
To prove the second statement we define $\tilde{\lambda}_0$ by $\dashint_{Q_R}\abs{\nabla u}^p=\tilde{\lambda}_0^2$. If $\tilde{\lambda}_0\leq 1$, then $\dashint_{Q_R}\abs{\nabla u}^p\leq 1^p$, in this case we define $\lambda_0=1$. If $\tilde{\lambda}_0\geq 1$ (and $\tilde{\lambda}_0^{2-p}\leq 1$), we define $\lambda_0=\tilde{\lambda}_0$ and find for any $Q_R^{\lambda_0}(t):=(t,t-\lambda_0^{2-p}R^2)\times B_R \subset Q_R$ that 
$\dashint_{Q_R^{\lambda_0}(t)}\abs{\nabla u}^p\leq \lambda_0^p$. Now we gain the result by proceeding as before.
\end{proof}

\begin{proposition}
\label{pro:intrbmo}
Let $Q_R^{\lambda_0}$ be sub-intrinsic, $\omega$ hold \eqref{eq:omega} and $g\in L^\infty(I,\setBMO_{\omega'}(B))$, with $\omega'\equiv \omega^{p-1}$. Then there exist a constant $c,\beta$ depending on $\gamma,c_1,n,p,N$ such that
\begin{align*}
 &\sup\limits_{z\in Q_\frac{R}2^{\lambda_0}}\sup\limits_{r<\frac{R}2}\frac{1}{\omega(r)}\bigg(\dashint_{Q_r^{\lambda_r}(z)}\abs{V(\nabla u)-\mean{V(\nabla u)}_{Q_{r}^{\lambda_r}(z)}}^2\bigg)^\frac1p\\
&\qquad\leq c\normgw{R}^\frac1{p-1}
+\frac{c\lambda_0}{\omega(R)}
\end{align*}
where $Q^{\lambda_{R/2}}_{R/2}(z)$ is defined by Lemma~\ref{lem:cubes} and $\Qr{}(z)\subset Q^{\lambda_{R/2}}_{R/2}(z)$ is defined by Lemma~\ref{lem:scal} for $\beta>\gamma$ fixed.
\end{proposition}
\begin{proof}
We fix $\rho :=\sup\{a<\frac R2|Q^{\lambda_{a}}_a(z)\text{ is intrinsic}\}$. By \ref{scal:6} of Lemma~\ref{lem:scal}, \eqref{eq:omega} and Lemma~\ref{lem:cubes} we find for $\rho\leq r\leq\frac R2$
 \[
  \frac 1{\omega^p(r)}\!\!\!\dashint_{ Q^{\lambda_{ r}}_{ r}(z)}\!\!\!\abs{V(\nabla u)-\mean{V(\nabla u)}_{Q^{\lambda_{ r}}_{ r}}(z)}^2\leq \!\frac{c\lambda_r^p}{\omega^p(r)}\leq c(\sigma)\frac{r^\beta\omega^p(R)}{R^\beta\omega^p(r)}\frac{\lambda_{R/2}^p}{\omega^p(R)}\leq \!\!\frac{c\lambda_0^p}{\omega^p(R)}.
 \]
For $r\leq \rho$ we can apply Lemma~\ref{pro:int} and find by the previous that
\begin{align*}
\frac{1}{\omega^p(r)}
&\dashint_{ \Qr{}(z)}\abs{V(\nabla u)-\mean{V(\nabla u)}_{\Qr{}(z)}}^2\\
&\leq c\normgw{R}^{p'}
+
\frac{c}{\omega^p(\rho)}\dashint_{ Q^{\lambda_{\rho}}_{\rho}(z)}\abs{V(\nabla u)-\mean{V(\nabla u)}_{Q^{\lambda_{\rho}}_{\rho}(z)}}^2\\
&\leq c\normgw{R}^{p'} +\frac{c\lambda_0^p}{\omega^p(R)}.
\end{align*}
This finishes the proof.
\end{proof}
We can generalize this result by the following purely intrinsic result
\begin{corollary}
 \label{cor:pureint}
Let $Q^{\lambda_0}_R$ be sub-intrinsic, $\omega$ hold \eqref{eq:omega} and for every cube $\Qr{}(z)$ constructed as in Proposition~\ref{pro:intrbmo}
\[
\sup\limits_{z\in Q_{R}^{\lambda_0}}
\sup\limits_{r<{R}}
\frac{1}{\omega^p(r)}\dashint_{\Qr{}(z)}\abs{g-\mean{g}_{\Qr{}(z)}}^{p'}
=:\normm{g}^{p'}<\infty,
\]
then
there exist a constant $c,\beta$ depending on $\gamma, c_1,n,p$ such that
\begin{align*}
 &\sup\limits_{z\in Q_\frac{R}2^{\lambda_0}}\sup\limits_{r<\frac{R}2}\frac{1}{\omega^p(r)}\dashint_{Q_r^{\lambda_r}(z)}\abs{V(\nabla u)-\mean{V(\nabla u)}_{Q_{r}^{\lambda_r}(z)}}^2\\
&\qquad\leq c\normm{g}^{p'}
+\frac{c\lambda_0^p}{\omega^p(R)}. 
\end{align*}
\end{corollary}
\begin{proof}
 One simply replaces $\norm{g}_{L^\infty(I,\setBMO(B_r))}$ by $\normm{g}$ in Lemma~\ref{eq:comparison}, Lemma~\ref{pro:comp} and Proposition~\ref{pro:intrbmo}. Anything else follows analogously.
\end{proof}
\subsection{Main Results}
We are now able to prove the main theorem on weighted BLO spaces. 
\begin{theorem}
 \label{thm:main} Let $p\geq 2$ and the wight $\omega:\setR^+\to\setR^+$ be almost increasing and satisfy \eqref{eq:omega}.
Let $u$ be a solution to \eqref{eq} on $I\times B$ and $g\in L^\infty(I,\setBMO_{\omega'}(B))$, with $\omega'\equiv\omega^{p-1}$, then $u\in L^\infty(I,\setBLO_{\omega}(B))$ locally. 
Moreover, there exists $c,\delta$ depending on $n,N,p,\gamma,c_1$ such that for every sub-intrinsic cylinder 
$Q_{R}^{\lambda_0}\subset I\times B$
\begin{align*} 
&\norm{u}_{L^\infty(I_{\lambda_0^{2-p}R^2/4},\setBLO_{\omega}(B_{\delta R/2}))}\\
&\qquad \leq \sup_{(t,x)\in  Q^{\lambda_0}_\frac{R}{2}}\sup\limits_{r\in(0,\frac{\delta R}2]}\frac1{\omega(r)}\bigg(\!\dashint_{B_{r}(x)}\Bigabs{\frac{u(t,y)-\ell_r(u)(t)}{r}}^2\dy\bigg)^\frac12\\
 &\qquad\leq c \normgw{R}^\frac1{p-1}
+\frac{c\lambda_0}{\omega(R)}.
\end{align*}
\end{theorem}
\begin{proof}
 We fix
$(t,x)\in Q^{\lambda_0}_{R/2}$ and construct $Q^{\lambda_{R/2}}_{R/2}(t,x)$ by Lemma~\ref{lem:cubes}. Then we define $Q^{\lambda_r}_r:=(t,t-\lambda_r^{2-p}r^2)\times B_r(x)\subset Q^{\lambda_0}_R$ for $\frac R2>r>0$ by Lemma~\ref{lem:scal} with respect to $Q^{\lambda_{R/2}}_{R/2}(t,x)$ for a convenient $\beta$. In the following all balls in space are centered in $(t,x)$. 
 Our aim is to estimate
\[
  N_r^\omega(u)(t,x):=\frac1{\omega^2(r)} \dashint_{ B_r(x)}\Bigabs{\frac{u(y,t)-\linear{u}{r}(t)}{r}}^2dy. 
\]
Here $\linear{u}{r}(t)$ is the best linear approximation of $u$ on $\set{t}\times B_r(x)$. We show the result for $\delta r\in (0,\frac{\delta R}2)$. The constant $\delta$ is fixed by Proposition~\ref{pro:comp}.
We will divide the proof in the two cases of Proposition~\ref{pro:comp}.

Case 1: $\lambda_r^p\leq K\normg^{p'}$.

Case 2: $\normg^{p'}\leq \frac1{K}\lambda_r^p$.

\noindent
If Case~1 holds, we find $\mean{\abs{\nabla u}^p}_{\Qr{}}\leq \lambda_r^p\leq \mu_r^p=:K\normg^{p'}$. As $\lambda_r\leq \mu_r$ we find that $ Q^{\mu_r}_r\subset \Qr{}$ and by \eqref{eq:mon} that $\mean{\abs{\nabla u}^p}_{Q^{\mu_r}_r}\leq \mu_r^p$. We take $h$ to be the solution of \eqref{eq:hom} on $Q^{\mu_r}_r$. Now Lemma~\ref{eq:comparison} gives
\begin{align}
\label{eq:mucomp}
  \mu_r^{p-2}\dashint_{\set{t}\times B_{r}}\Bigabs{\frac{u-h}{r}}^2dx+\dashint_{Q^{\mu_r}_r}\abs{V(\nabla u)-V(\nabla h)}^2\dz\leq c\mu_r^p.
\end{align}
We use that $\linear{u}{\delta r}$ is the best linear approximation of $u$ on $B_{\delta r}:=\set{t}\times B_{\delta r}(x)$ and \Poincare's inequality to gain
\begin{align}
\label{eq:case1}
\begin{aligned}
  N^\omega_{\delta r}(u) &\leq 
\frac{1}{\omega^2(\delta r)}\dashint_{B_{\delta r}}\Bigabs{\frac{u-\linear{h}{\delta r}}{r}}^2dx 
\leq \frac{c}{\omega^2(\delta r)}\dashint_{B_{\delta r}}\Bigabs{\frac{u-h}{r}}^2+\Bigabs{\frac{h-\linear{h}{\delta r}}{r}}^2dx\\
&\leq \frac{c(\delta)}{\omega^2(r)} \dashint_{B_{r}}\Bigabs{\frac{u-h}{r}}^2dx
+\frac{c}{\omega^2(\delta r)}\sup\limits_{ B_{\delta r}}\abs{\nabla h}^2 = I+II.
\end{aligned}
\end{align}
For $I$ we find by \eqref{eq:mucomp}
\[
 I\leq \frac{c(\delta)}{\omega^2(r)}\mu_r^2\leq c \normgw{r}^\frac{2}{p-1}.
\]
To estimate $II$ we find by \eqref{eq:mucomp}
\begin{align*}
\mean{\abs{\nabla h}^p}_{Q^{\mu_r}_r}\leq \mean{\abs{\nabla u}^p}_{Q^{\mu_r}_r}+\dashint_{Q^{\mu_r}_r}\abs{V(\nabla u)-V(\nabla h)}^2\dz\leq c\mu_r^p.
\end{align*}
 Now Theorem~\ref{thm:sup} implies
\[
\frac{c}{\omega^2(\delta r)} \sup\limits_{\set{t}\times B_{\delta r}(x)}\abs{\nabla h}^2\leq \frac{c}{\omega^2(\delta r)} \sup\limits_{Q^{\mu_r}_{\delta r}}\abs{\nabla h}^2\leq\frac{c}{\omega^2(r)}\mu_r^2\leq c\normgw{R}^\frac{2}{p-1}.
\]
This closes Case 1.

In the following Case 2 holds. Remember, that $\delta r\in (0,\delta \frac{R}{2})$. We start similar to \eqref{eq:case1}. we take $h$ to be the solution  of \eqref{eq:hom} on $\Qr{}$. Now Lemma~\ref{eq:comparison} gives
\begin{align}
\label{eq:lcomp}
  \lambda_r^{p-2}\dashint_{\set{t}\times B_{r}}\Bigabs{\frac{u-h}{r}}^2dx+\dashint_{\Qr{}}\abs{V(\nabla u)-V(\nabla h)}^2\dz\leq c\normg^{p'}.
\end{align}
Similar to Case ~1 we find
\begin{align*}
 N^\omega_{\delta r}(u) &\leq 
\frac{1}{\omega^2(\delta r)}\dashint_{B_{\delta r}}\Bigabs{\frac{u-\linear{h}{\delta r}}{r}}^2dx 
\leq \frac{c}{\omega^2(\delta r)}\dashint_{B_{\delta r}}\Bigabs{\frac{u-h}{r}}^2+\Bigabs{\frac{h-\linear{h}{\delta r}}{r}}^2dx\\
&\leq  \frac{c(\delta)}{\omega^2(r)}\dashint_{B_{r}}\Bigabs{\frac{u-h}{r}}^2 dx+ \frac{c}{\omega^2(\delta r)}\osc{\nabla h}{B_{\delta r}}^2= I+II.
\end{align*}
where we used the \Poincare's inequality. 
 We estimate $I$ by Lemma~\ref{eq:comparison}; as Case 2 holds we deduce from \eqref{eq:lcomp}
\begin{align*}
 \dashint_{B_{r}}\Bigabs{\frac{u-h}{r}}^2 dx\leq \lambda_r^{2-p}\normg^{p'}\leq \normg^\frac{2}{p-1},
\end{align*}
and consequently
\begin{align}
 \label{eq:I}
I\leq c\normgw{r}^\frac2{p-1}.
\end{align}
We estimate $II$ by using $p\geq2$ and Proposition~\ref{pro:comp}. As in the proof of Proposition~\ref{pro:intrbmo} we fix
 $\rho :=\sup\{a<\frac R2|Q^{\lambda_{a}}_a(t,x)\text{ is intrinsic}\}$. If $r\leq \rho$ Proposition~\ref{pro:comp} provides an ${r_1}\leq \rho$ such that 
\[
 \osc{V(\nabla h)}{B_{\delta r}}^2\leq c\Big(\frac{\delta r}{{r_1}}\Big)^\beta\!\!\!\!\dashint_{Q_{{r_1}}^{\lambda_{r_1}}}
\abs{V(\nabla u)-\mean{V(\nabla u)}_{Q_{{r_1}}^{\lambda_{{r_1}}}}}^2\dz+c\delta^\beta \normg^{p'}.
\]
This implies using \eqref{eq:nervig}
\begin{align*}
 II&=\frac{c}{\omega^2(\delta r)}\osc{\nabla h}{B_{\delta r}}^2 \leq  \frac{c}{\omega^2(\delta r)}\osc{V(\nabla h)}{B_{\delta r}}^\frac{4}{p} \\
&\leq c
\frac{1}{\omega^2(\delta r)}\bigg(\Big(\frac{\delta r}{{r_1}}\Big)^{\beta}\!\!\!\!\dashint_{Q_{{r_1}}^{\lambda_{r_1}}}
\abs{V(\nabla u)-\mean{V(\nabla u)}_{Q_{{r_1}}^{\lambda_{{r_1}}}}}^2\dz +\normg^{p'}\bigg)^\frac{2}{p}
\\
 &\leq c\Big(\frac{\delta r}{{r_1}}\Big)^{\frac{2}{p}(\beta-\gamma)}\bigg(\frac{1}{\omega^{p}({r_1})}\dashint_{Q_{{r_1}}^{\lambda_{r_1}}}
\abs{V(\nabla u)-\mean{V(\nabla u)}_{Q_{{r_1}}^{\lambda_{{r_1}}}}}^2\dz\bigg)^\frac2p\\
&\quad+c\normgw{r}^\frac{2}{p-1}
\end{align*}
as $\omega$ holds \eqref{eq:omega}. On this we can apply Proposition~\ref{pro:intrbmo} and find as $\gamma< \beta$
\begin{align}
\label{eq:II}
 II\leq c\normgw{R}^\frac{2}{p-1}+\frac{c\lambda_0^2}{\omega^2(R)}.
\end{align}
If $\rho<r<\frac R4$ we have by \ref{scal:6} of Lemma~\ref{lem:scal} and the construction of $Q^{\lambda_{R/2}}_{R/2}(t,x)$, that $\lambda_r\leq \Big(\frac{r}{R/2}\Big)^\beta \lambda_0$ and therefore we find by \eqref{eq:lcomp} and it's consequences \eqref{eq:hsup} and \eqref{eq:decay2}
\[
 II\leq \frac{c}{\omega^2(\delta r)}\osc{\nabla h}{B_{\delta r}}^2\leq \frac{c}{\omega^2(\delta r)}\lambda_r^2\leq  \frac{c}{\omega^2(R)}\Big(\frac{r}{R}\Big)^{\beta-\gamma}\lambda_0^2.
\]
Combining the last estimate with \eqref{eq:I} and \eqref{eq:II} closes case~2.
As all estimates are independent of $(t,x)\in Q^{\lambda_0}_\frac{R}{2}$, the result is proved.
\end{proof} 
\begin{proof}[Proof of Theorem~\ref{thm:m}]
 One fixes $\omega(r)\equiv 1$ and combines Lemma~\ref{lem:cubes} with Theorem~\ref{thm:main}. Then the result follows by the Campanato characterization of $\mathcal{C}^1(B_{R/2}(x))$.
\end{proof}

\begin{remark}
 \label{rem:triebel}
In \cite[Section: 1.7.2]{Tri92} we find that
BLO $=\mathcal{C}^1=F^1_{\infty,\infty}$, here $F^1_{\infty,\infty}$ is the Triebel-Lizorkin space. The space
$W^{1,\setBMO}=F^1_{\infty,2}$. We therefore find by our estimates that, if $g$ is in $ L^\infty(2I,\setBMO(2B))$, then we have
$u\in L^p(I,W^{1,p}(B))\cap L^\infty(I,BLO(B))=L^p(I,F^1_{p,2}(B))\cap L^\infty(I,F^1_{\infty,\infty}(B))$. By interpolation $u\in L^q(I,W^{1,r}(B))$ for every $1\leq q\leq\infty,$ $1\leq r<\infty$ (see \cite[Section: 1.6.2]{Tri92}); natural local estimates are available.   
\end{remark}

\begin{proposition}
 \label{cor:hoelder}
Let $\gamma p<\min\Bigset{\frac{\alpha}{1+\alpha\frac{p-2}2},\frac2{p-2}}$.
If $g\in L^\infty(I,C^{\gamma(p-1)}(B))$, then $\nabla u\in C^\gamma_{\text{par}}(I\times B)$. Moreover, for every sub-intrinsic cylinder $Q^{\lambda_0}_R$ we find
\begin{align*}
 \norm{\nabla u}_{C^\gamma_{\text{par}}(Q^{\lambda_0}_{R/4})}\leq c\Big(\frac{1}{R^{\gamma}}+\frac{1}{(K^{2-p}R^2)^\frac{\gamma}{2}}\Big),
\end{align*}
where $K=c \lambda_0+ cR^\gamma \norm{g}_{L^\infty(I,C^{\gamma(p-1)}(B_R))}^\frac1{p-1}$ and $c$ depends on $\gamma,n,p,K$.
\end{proposition}
\begin{proof}
 We start by showing H\"older continuity in space. By Theorem~\ref{thm:main} and $\omega(r)=r^{\gamma}$ we gain by the Campanato characterization that
\begin{align}
\label{eq:hoel}
\norm{\nabla u}_{L^\infty(I_{\lambda_0^{2-p}R^2/4},C^{\gamma}( B_{R/2}))}\leq c\norm{g}_{L^\infty(I,C^{\gamma(p-1)}(B_R))}^\frac1{p-1}+\frac{c\lambda_0}{R^\gamma}. 
\end{align}
 This implies that $\nabla u$ is H\"older continuous in space. It implies also, that $\nabla u$ is bounded in $Q^{\lambda_0}_{R/2}$. Moreover, the previous implies 
\[\max_{Q^{\lambda_0}_{R/2}}\abs{\nabla u}\leq K<\infty\]
for $K=c \lambda_0+ cR^\gamma \norm{g}_{L^\infty(I,C^{\gamma(p-1)}(B_R))}^\frac1{p-1}$.

In the following we prove H\"older continuity in time. I.e.\ we show for  $(t,x)\in Q^{\lambda_0}_{R/4}$,
\begin{align}
\label{eq:timehoelder1}
 \bigg(\dashint^t_{t-s}&\abs{V(\nabla u)(\tau,x)-\mean{V(\nabla u)(\tau,x)}_{(t,t-s)}}^2\dtau\bigg)^\frac1p \leq  K\Big(\frac{s}{S}\Big)^\frac{\gamma}2.
\end{align}
for all $s\in (0,S)$, $S:=K^{p-2}R^2$. From this estimate the H\"older continuety in time follows by \eqref{eq:nervig} and the Campanato characterization of H\"older spaces.

In the following we prove \eqref{eq:timehoelder1}. We take $(t,x)\in Q^{\lambda_0}_{R/4}$, fix $S(R)=K^{2-p}R$ and take $Q^K_{R/4}(t,x)\subset Q^{\lambda_0}_{R/2}$ as starting cylinder.  Then for all $r<\frac R4$ we take $Q^{\lambda_r}_r(t,x)$ constructed by Lemma~\ref{lem:scal}. We have that $\lambda_r\leq K$, (as $\tilde{\lambda}\leq K$ by \eqref{eq:sleqS}). 
 Therefore Proposition~\ref{pro:intrbmo} provides for all $r\in (0,\frac{R}{4}]$
\begin{align}
\label{eq:timehoelder}
\begin{aligned}
\dashint^t_{t-s(r)}&\dashint_{B_r(x)}\abs{V(\nabla u)-\mean{V(\nabla u)}_{\Qr{}(z)}}^2\\
&\leq c\Big(\frac{r}{R}\Big)^{p\gamma}K^p
=c\Big(\frac{\lambda_r}{K}\Big)^\frac{p-2}2\Big(\frac{s(r)}{S(R)}\Big)^\frac{\gamma p}2 K^p
\leq c\Big(\frac{s(r)}{S(R)}\Big)^\frac{\gamma p}2  K^p,
\end{aligned}
\end{align}
as $\lambda_r\leq K$. Now we find by Lemma~\ref{lem:scal}, \ref{scal:0}, that $s(r)=\lambda_r^{2-p}r^2$ is continuous and $s(0)=0$ and $s(R)=S(R)$. Therefore we can choose an $r(s)$ for every $0<s\leq S$ such that $s=s(r)=\lambda_{r(s)}^{2-p}r^2(s)$.
We estimate for $x\in B_{R/4}$ and $s$ fixed
\begin{align*}
\dashint^t_{t-s}&\abs{V(\nabla u)(\tau,x)-\mean{V(\nabla u)(\tau,x)}_{(t,t-s)}}^2\dtau
\\
&\leq c\dashint^t_{t-s}\abs{V(\nabla u)(\tau,x)-\mean{V(\nabla u)}_{Q^{\lambda_{r(s)}}_{r(s)}}}^2\dtau\\
&\leq c\dashint^t_{t-s}\abs{V(\nabla u)(\tau,x)-\mean{V(\nabla u)(\tau)}_{B_{r(s)}(x)}}^2\dtau\\
&\qquad +c\dashint^t_{t-s}\abs{\mean{V(\nabla u)(\tau)}_{B_{r(s)}(x)}-\mean{V(\nabla u)}_{Q^{\lambda_{r(s)}}_{r(s)}}}^2\dtau=I+II.
\end{align*}
$I$ can be estimated by the $L^\infty(I_{R^2/4},C^{1,\gamma}(B_{\frac R2}))$ estimate 
\begin{align*}
&I\leq K^p\Big(\frac{r(s)}{R(S)}\Big)^{p\gamma} \leq K^p\Big(\frac{s}{S}\Big)^\frac{\gamma p}2.
\end{align*}
$II$ can be estimated by \eqref{eq:timehoelder}
\begin{align*}
II\leq \dashint^t_{t-s}\dashint_{B_{r(s)}}&\abs{V(\nabla u)-\mean{V(\nabla u)}_{Q^{\lambda_{r(s)}}_{r(s)}}}^2
\leq c\Big(\frac{s}{S}\Big)^\frac{\gamma p}2K^p,
\end{align*}
where we used that $Q^{\lambda_{r(s)}}_{r(s)}=(t,t-s(r))\times B_{r(s)}(x)$.
 This finishes the proof of \eqref{eq:timehoelder1}.
\end{proof}

\begin{remark}
\label{rem:cont}
 The last result can be weakened. As long as the modulus of continuity is strong enough to imply the boundedness of $\abs{\nabla u}$ we find the same natural estimates as in Proposition~\ref{cor:hoelder}. We expect that the sharp bound would be the Dini continuity. I.e.\ $f$ is Dini continuous on $B_R$ if it's modulus of continuity $\omega$ holds $\sum_{i=1}^\infty\omega(2^{-i}R)<\infty$. We conjecture that in this case $\setBLO_{\omega}\equiv C^{1,\omega}$. If this would be true, then the Dini result of \cite{KuuMin12} could be gained similar to Proposition~\ref{cor:hoelder} with a weaker condition on $g$, i.e.\ $g\in L^\infty(I,C_{\omega(p-1)}(B))$, but restricted to \eqref{eq} and $p\geq 2$.

If we follow the estimates of \cite[Corr. 5.4]{DieKapSch11}, we find directly, that $g\in L^\infty(I,\text{VMO}(B))$ implies that locally $u \in L^\infty(I,\text{VLO}(B))$.
\end{remark}
\begin{appendix}
\section{}
%
 For $Q_1\subset Q_2$ and $q\in[1,\infty)$ we find that
\begin{align}
 \label{eq:means}
\abs{\mean{f}_{Q_1}-\mean{f}_{Q_2}}
\leq 
\bigg(\dashint_{Q_1}\abs{f-\mean{f}_{Q_2}}^q\bigg)^\frac1q
\leq
\bigg(\frac{\abs{Q_2}}{\abs{Q_1}}\dashint_{Q_2}\abs{f-\mean{f}_{Q_2}}^q\bigg)^\frac1q.
\end{align}
This estimate can be iterated for $i=\set{0\dots k}$ and $Q_i\subset Q_{i-1}$ with $\frac{\abs{Q_{i-1}}}{\abs{Q_i}}\leq c$ 
\begin{align}
\label{eq:meanit}
 \abs{\mean{f}_{Q_k}-\mean{f}_{Q_0}}\leq \sum_{i=1}^k\abs{\mean{f}_{Q_i}-\mean{f}_{Q_{1-i}}}
\leq c\sum_{i=1}^k\bigg(\dashint_{Q_{i-1}}\abs{f-\mean{f}_{Q_{i-1}}}^q\bigg)^\frac1q.
\end{align}
\begin{lemma}
  \label{lem:osc}
Let $Q_1\subset Q$ be two Cylinders and $f\in L^q(Q)$ for $q\in[1,\infty)$. For $\epsilon\in (0,1)$ we find:

  If $ \abs{\mean{f}_{Q_1}}\leq \epsilon \mean{\abs{f}^q}_{Q}^\frac1q$, then
  \[
  \abs{\mean{f}_{Q_1}}\leq \epsilon\mean{\abs{f}^q}_{Q}^\frac1q\leq \frac{\epsilon}{1-\epsilon}\Big(1+\Big(\frac{\abs{Q}}{\abs{Q_1}}\Big)^{\frac 1q}\Big)\bigg(\dashint_{Q}\abs{f-\mean{f}_{Q}}^q\dx\bigg)^\frac1q.
  \]
\end{lemma}

\begin{proof}
We find
\begin{align*}
    \mean{\abs{f}^q}_{Q}^\frac1q
&\leq \bigg(\dashint_{Q}\abs{f-\mean{f}_{Q_1}}^q\dx\bigg)^\frac1q+\abs{\mean{f}_{Q_1}}\\
&\leq \bigg(\dashint_{Q}\abs{f-\mean{f}_{Q}}^q\dx\bigg)^\frac1q +\abs{\mean{f}_{Q}-\mean{f}_{Q_1}}+\epsilon\mean{\abs{f}^q}_{Q}^\frac1q
  \end{align*}
 This implies that
  \begin{align*}
      \mean{\abs{f}^q}_{Q}^\frac1q&\leq 
\frac{1}{1-\epsilon}
\bigg(\dashint_{Q}\abs{f-\mean{f}_{Q}}^q\dx\bigg)^\frac1q+\frac1{1-\epsilon}\abs{\mean{f}_{Q_1}-\mean{f}_{Q}}
\end{align*}
We estimate the second integral by
\begin{align*} 
\abs{\mean{f}_{Q_1}-\mean{f}_{Q}}
&\leq \dashint_{Q_1}\abs{f-\mean{f}_{Q}}\dx
\leq 
\bigg(\dashint_{Q_1}\abs{f-\mean{f}_{Q}}^q\dx\bigg)^\frac1q\\
&\leq \bigg(\frac{\abs{Q_1}}{\abs{Q}}\dashint_{Q}\abs{f-\mean{f}_{Q}}^q\dx\bigg)^\frac1q.
  \end{align*} 
\end{proof}
 \begin{lemma}
  \label{lem:osc2}
Let $f\in L^q(Q_R)$ with $q\in[1,\infty)$. Suppose that $\omega:\setR^+\to\setR^+$ is increasing and holds the following Dini condition: $\sum_i^\infty\omega(2^{-i}R)\leq K$ (e.g. $\omega(r)=r^\gamma$). If
\[
 \bigg(\dashint_{\theta B}\abs{f-\mean{f}_{\theta B}}^q\bigg)^\frac1q
\leq
 c_1\omega(\theta)\bigg(\dashint_B\abs{f-\mean{f}_{B}}^q\bigg)^\frac1q,
\]
then
\[
 \osc{f}{\theta Q_\rho}\leq cK\omega(\theta)\bigg(\dashint_{Q_\rho}\abs{f-\mean{f}_{Q_R}}^q\bigg)^\frac1q,
\]
for all $\theta\in(0,\frac12)$, $\rho\leq R$ and $c$ depending only on $q,n, c_1$.
\end{lemma}
%

\begin{proof}
We only proof the first statement. 
 For $k\in\setN$ we define for $z\in \frac12\theta Q_\rho$ we define $Q_i(z):= 2^{-i}\frac12Q_{\theta \rho}(z)$ for $i=1,...,k$ and $Q_0(z)=\theta Q_\rho$. We estimate by \eqref{eq:meanit}
\[
\abs{\mean{f}_{Q_k(z)}-\mean{f}_{\theta Q_\rho}}
\leq
 \sum_{i=0}^{k-1}\bigg(\dashint_{Q_i(z)}\abs{f-\mean{f}_{Q_{i}(z)}}^q\bigg)^\frac1q
 \]
this can be estimated by assumption by and because $\omega$ is increasing
\begin{align*}
 \abs{\mean{f}_{Q_k(z)}-\mean{f}_{\theta Q_\rho}}&\leq c\sum_{i=1}^{k-1}\omega(2^{-i}\theta \rho)
\bigg(\dashint_{\theta Q_\rho}\abs{f-\mean{f}_{Q_{\theta \rho}}}^q\bigg)^\frac1q\\
&\leq
 cK\omega(\theta)\bigg(\dashint_{ Q_\rho}\abs{f-\mean{f}_{Q_{\rho}}}^q\bigg)^\frac1q;
\end{align*}
the constant is independent of $k$; this implies that
\[
 \abs{f(z)-\mean{f}_{\theta Q_\rho}}\leq cK\omega(\theta)\bigg(\dashint_{ Q_\rho}\abs{f-\mean{f}_{Q_{\rho}}}^q\bigg)^\frac1q.
\]
Consequently, we find for $z,w\in \frac12\theta Q_\rho$ we have 
\[
 \abs{f(z)-f(w)}\leq cK\omega(\theta)\bigg(\dashint_{ Q_\rho}\abs{f-\mean{f}_{Q_{\rho}}}^q\bigg)^\frac1q.
\]
\end{proof}
\end{appendix}

\bibliographystyle{abbrv} 

\begin{thebibliography}{10}

\bibitem{AceMin07}
E.~Acerbi and G.~Mingione.
\newblock Gradient estimates for a class of parabolic systems.
\newblock {\em Duke Math. J.}, 136(2):285--320, 2007.

\bibitem{DiB93}
E.~DiBenedetto.
\newblock {\em Degenerate parabolic equations}.
\newblock Springer-Verlag, New York, 1993.

\bibitem{DibFri85}
E.~DiBenedetto and A.~Friedman.
\newblock H\"older estimates for nonlinear degenerate parabolic systems.
\newblock {\em J. Reine Angew. Math.}, 357:1--22, 1985.

\bibitem{DiBMan93}
E.~DiBenedetto and J.~Manfredi.
\newblock On the higher integrability of the gradient of weak solutions of
  certain degenerate elliptic systems.
\newblock {\em Amer. J. Math.}, 115(5):1107--1134, 1993.

\bibitem{DieE08}
L.~Diening and F.~Ettwein.
\newblock Fractional estimates for non-differentiable elliptic systems with
  general growth.
\newblock {\em Forum Mathematicum}, 20(3):523--556, 2008.

\bibitem{DieKapSch11}
L.~Diening, P.~Kaplick{\'y}, and S.~Schwarzacher.
\newblock B{MO} estimates for the {$p$}-{L}aplacian.
\newblock {\em Nonlinear Anal.}, 75(2):637--650, 2012.

\bibitem{DieKapSch13}
L.~Diening, P.~Kaplick{\'y}, and S.~Schwarzacher.
\newblock Campanato estimates for the generalized stokes system.
\newblock {\em Annali di Matematica Pura ed Applicata}, 2013.

\bibitem{GiaM86}
M.~Giaquinta and G.~Modica.
\newblock Remarks on the regularity of the minimizers of certain degenerate
  functionals.
\newblock {\em Manuscripta Math.}, 57(1):55--99, 1986.

\bibitem{Iwa83}
T.~Iwaniec.
\newblock Projections onto gradient fields and {$L^{p}$}-estimates for
  degenerated elliptic operators.
\newblock {\em Studia Math.}, 75(3):293--312, 1983.

\bibitem{JohNir61}
F.~John and L.~Nirenberg.
\newblock On functions of bounded mean oscillation.
\newblock {\em Comm. Pure. Appl. Math}, 14:415--426, 1961.

\bibitem{KuuMin12}
T.~Kuusi and G.~Mingione.
\newblock New perturbation methods for nonlinear parabolic problems.
\newblock {\em J. Math. Pures Appl. (9)}, 98(4):390--427, 2012.

\bibitem{KuuMin122}
T.~Kuusi and G.~Mingione.
\newblock Potential estimates and gradient boundedness for nonlinear parabolic
  systems.
\newblock {\em Rev. Mat. Iberoam.}, 28, 2012.

\bibitem{KuuMin13}
T.~Kuusi and G.~Mingione.
\newblock Linear potentials in nonlinear potential theory.
\newblock {\em Arch. Ration. Mech. Anal.}, 207(1):215--246, 2013.

\bibitem{Mis02}
M.~Misawa.
\newblock Local {H}\"older regularity of gradients for evolutional
  {$p$}-{L}aplacian systems.
\newblock {\em Ann. Mat. Pura Appl. (4)}, 181(4):389--405, 2002.

\bibitem{Mis13}
M.~Misawa.
\newblock A {H}\"older estimate for nonlinear parabolic systems of
  {$p$}-{L}aplacian type.
\newblock {\em J. Differential Equations}, 254(2):847--878, 2013.

\bibitem{Tri92}
H.~Triebel.
\newblock {\em Theory of function spaces. {II}}, volume~84 of {\em Monographs
  in Mathematics}.
\newblock Birkh\"auser Verlag, Basel, 1992.

\bibitem{Zyg45}
A.~Zygmund.
\newblock Smooth functions.
\newblock {\em Duke Math. J.}, 12:47--76, 1945.

\end{thebibliography}

\end{document}